\documentclass[envcountsect,smallextended]{svjour3} 

\smartqed 
\usepackage{graphicx}
\journalname{JOTA}

\usepackage{cite} 

\usepackage{fullpage}

\usepackage{amsmath}
\usepackage{amssymb}
\usepackage{subfigure}
\usepackage{graphicx}
\usepackage{epstopdf}
\usepackage{algorithmic}
\usepackage{algorithm}
\usepackage{multirow,array}

\usepackage{hyperref,color}

\newtheorem{assumption}[theorem]{Assumption}

\let\la\langle
\let\ra\rangle
\newcommand\rnr{\reals^{n\times r}}
\newcommand\rnn{\reals^{n\times n}}

\newcommand{\BEAS}{\begin{eqnarray*}}
\newcommand{\EEAS}{\end{eqnarray*}}
\newcommand{\BEA}{\begin{eqnarray}}
\newcommand{\EEA}{\end{eqnarray}}
\newcommand{\BEQ}{\begin{equation}}
\newcommand{\EEQ}{\end{equation}}
\newcommand{\BIT}{\begin{itemize}}
\newcommand{\EIT}{\end{itemize}}
\newcommand{\BNUM}{\begin{enumerate}}
\newcommand{\ENUM}{\end{enumerate}}

\newcommand{\BA}{\begin{array}}
\newcommand{\EA}{\end{array}}

\newcommand{\reals}{{\mathbb R}}

\newcommand{\Tr}{\mathop{\bf Tr}}
\newcommand{\diag}{\mathop{\bf diag}}

\newcommand{\argmin}{\mathop{\rm argmin}}

\newcommand{\dom}{\mathop{\bf dom}}

\begin{document}

\title{Quartic First-Order Methods for Low-Rank Minimization}

\subtitle{}

\author{Radu-Alexandru Dragomir \and  Alexandre d'Aspremont \and J\'er\^ome Bolte }

\institute{Radu-Alexandru Dragomir, corresponding author \at
             Universit\'e Toulouse 1 Capitole \& D.I. \'Ecole Normale Sup\'erieure, Paris, France \\
             radu-alexandru.dragomir@inria.fr
           \and
           Alexandre d'Aspremont \at
           CNRS \& D.I. \'Ecole Normale Sup\'erieure, Paris, France \\
            aspremon@ens.fr.
            \and
            J\'er\^ome Bolte \at
            Universit\'e Toulouse 1 Capitole, Toulouse, France \\
            jerome.bolte@ut-capitole.fr
}

\date{Last revised on January 18, 2021}

\maketitle

\begin{abstract}
We study a general nonconvex formulation for low-rank minimization problems. We use recent results on non-Euclidean  first-order methods to provide efficient and scalable algorithms. Our approach uses the geometry induced by the Bregman divergence of well-chosen kernel functions; for unconstrained problems we introduce a novel family of Gram quartic kernels that improve numerical performance. 

Numerical experiments on Euclidean distance matrix completion and symmetric nonnegative matrix factorization show that our algorithms scale well and reach state of the art performance when compared to specialized methods.
\end{abstract}


\section{Introduction}

We consider the problem of minimizing a smooth convex function over the set of low-rank positive semidefinite matrices. Fundamental applications of this problem arise in various areas of data analysis including matrix completion \cite{Cand2008,Completion2010,Jain2012}, matrix sensing \cite{Recht2007}, Euclidean matrix completion \cite{Mishra2011,Fang2012}, phase retrieval \cite{Candes2015}, robust principal component analysis \cite{Chen2015}, to name a few. 

A popular approach to low-rank semidefinite minimization, known as the Burer-Monteiro formulation \cite{Burer2005}, consists in explicitly modeling the rank constraint by writing the matrix in a factorized form. This method is especially appealing for large-scale instances, since it requires storing much less variables than the standard semidefinite programming approaches; see \cite{Tu2015,Bhojanapalli2015,Zhao2015,Sun2015,Chen2015,Zheng2016,Park2016} and references therein. 

This formulation comes however with an important drawback, as the problem becomes nonconvex, even if the original objective is convex. Therefore, local optimization methods can generally only hope to find a stationary point, or at best a local minimum. Nevertheless, recent work shows convergence towards a global optimum for a close enough initialization \cite{Tu2015,Bhojanapalli2015,Park2016}, or under additional statistical assumptions about the problem \cite{Chen2015,Zheng2015,Ge2016}. Although these global optimality results often impose restrictive assumptions that may not be satisfied in practice, they help to explain why using local algorithms to solve Burer-Monteiro problem formulations often leads to satisfactory solutions in practice.

The most commonly used algorithm to solve these problem formulations is some variant of the proximal gradient method. However, a critical issue with gradient schemes is the choice of step sizes, which significantly impacts performance. This step size choice is closely related to the smoothness of the objective. In particular, when it has a $L$-Lipschitz continuous gradient with respect to the Euclidean norm, standard gradient methods can be applied with a step size lying in $]0,1/L]$. This smoothness assumption is used in the broad majority of theoretical analyses of gradient algorithms, yet there are many cases where it is not satisfied \cite{Bauschke2017,Bolte2018}. In particular, it does not hold for the general Burer-Monteiro low-rank problem, as we will show in what follows.

Of course, there is a way to circumvent this issue in classical Euclidean methods, by using an Armijo line search \cite{Lin2007}. However, in some cases, this naive line search strategy generates very small step sizes which in turn involve costly subroutines. Other approaches impose a step size that is only proven to be valid in a small neighborhood of the optimum \cite{Bhojanapalli2015,Park2016}. 

\paragraph{Non-Euclidean gradient methods.} We adopt an original approach  based on a recent line of work on non-Euclidean gradient methods \cite{Bauschke2017,Bolte2018,van} and subsequent work \cite{Lu2016}. {Unlike standard gradient descent that uses the uniform Euclidean geometry, the NoLips method, also known as Bregman/Mirror descent, uses the Bregman divergence induced by a well-chosen convex \textit{kernel} function. This allows the algorithm to take gradient steps that are more adapted to the geometry of the problem, advancing faster in directions where the gradient of the objective changes slowly, thus improving convergence speed. 
The kernel function is chosen so that the objective function satisfies a compatibility condition called \textit{relative smoothness} \cite{Bauschke2017,Lu2016}, which is a generalization of the usual smoothness assumption mentioned earlier.}

In our setting, the objective has a quartic growth, hence choosing the geometry induced by a quartic polynomial will prove to be efficient.

\paragraph{Contributions.}
In this work, we focus on deriving efficient algorithms to find stationary points of nonconvex low-rank problems. 
Our main contribution is to identify favorable non-Euclidean geometries for these problems, induced by well-chosen quartic kernels.

We first study a simple quartic \textit{norm} kernel that is compatible with various regularization terms. We then introduce a novel family of quartic kernels that we call \textit{Gram kernels}, which can be applied to unregularized problems. They provide richer geometries which greatly improve convergence speed with little impact on the iteration complexity. We also extend the NoLips scheme to Dyn-NoLips, allowing for adaptive step size strategies. 
 
To highlight the benefits of our approach, we study applications to symmetric nonnegative matrix factorization and Euclidean distance matrix completion and show competitive numerical performance compared to specialized algorithms for these problems.

\paragraph{Notations} For a square matrix $M$, we denote its trace $\Tr M = \sum_{i=1}^n M_{ii}$. For two matrices $X$ and $Y$ of same size, we denote the standard Euclidean inner product and norm by $\la X,Y \ra = \Tr(X^T Y)$ and $\|X\| = \sqrt{ \Tr(X^T X) }$. For a function $f:\reals^{n \times r} \rightarrow \reals$, we denote by $\nabla F(X)$ its gradient matrix $\nabla F(X)_{ij} = \frac{\partial F(X)}{\partial x_{ij}}$ and by $\nabla^2 F(X)[U,V]$ the second derivative at $X$ in the directions $U,V \in \reals^{n\times r}$. 
$I_r$ denotes the identity matrix of size $r\times r$. For two square matrices $X,Y$, we write $X \preceq Y$ if the matrix $Y-X$ is positive semidefinite. We write $\|\mathcal{A}\|_{\text{op}}$ for the operator norm of a linear application $\mathcal{A}$.

\section{Quartic Geometries for Low-Rank Minimization}\label{s:geometries}
\subsection{Problem Setup}

Let $n \geq 1$ and consider a low-rank semidefinite program, written
\begin{equation}\label{eq:semidef_problem}
\tag{SDP-r}
    \min F(Y) \quad \text{subject to } \, Y \succeq 0, \, \text{rank}(Y) \leq r
\end{equation}
in the variable $Y \in \reals^{n \times n}$, where $F$ is a smooth convex function and $r \leq n$ is the target rank.
The Burer-Monteiro formulation \cite{Burer2005} consists in representing $Y$ as $Y = XX^T$ to solve instead
\begin{equation}\label{eq:nolips_min_problem}
\tag{P}
    \min \Psi (X) := F(XX^T) + g(X)
\end{equation}
in the variable $X \in \reals^{n\times r}$, where $g$ is a \textit{simple} convex regularization function. $F$ is typically a quadratic loss function, and $g$ enforces penalties on the factor $X$ such as sparsity when choosing the $\ell_1$ norm, or nonnegativity when choosing the indicator function of the nonnegative orthant.
We will write $f:\reals^{n \times r} \rightarrow \reals$ the factorized function defined by \[f(X) := F(XX^T).\] 
Throughout the paper, we make the following standing assumptions.
\begin{assumption}\label{ass:standing}

\begin{enumerate}
    \item[\rm (a)] $F:\reals^{n\times n} \rightarrow \reals$ is a twice continuously differentiable function which is {$\mu_F$-strongly convex} and $L_F$-smooth, i.e.,
    \begin{align*}
    \la \nabla F(X) - \nabla F(Y), X-Y \ra & \geq \mu_F \|X-Y\|^2, \\
    \|\nabla F(X) - \nabla F(Y)\| &\leq L_F \|X-Y\| \quad \forall X,Y \in \reals^{n \times n},
    \end{align*}
    \item[\rm (b)] $g :\reals^{n\times n} \rightarrow \reals \cup \{+\infty\}$ is a closed convex proper function,
    \item[\rm (c)] $\min_{\rnr} \Psi > -\infty$.
\end{enumerate}
\end{assumption}
{Our analysis will involve the following lemma.
\begin{lemma}\label{lemma:muf}
    Let $F:\reals^{n \times n}\rightarrow \reals$ be a twice differentiable $\mu_F$-strongly convex and $L_F$-smooth function. Then, the function $G:=F-\frac{\mu_F}{2}\|\cdot\|^2$ is convex and $(L_F - \mu_F)$-smooth.
\end{lemma}
\begin{proof}
It suffices to use the second-order characterization \cite{Nesterov2004} and notice that, for ${Y,U \in \reals^{n \times n}}$, we have $\nabla^2 G(Y)[U,U] =\nabla^2 F(Y)[U,U] - \mu_F\|U\|^2$ and hence
\[ \mu_F \|U\|^2 \leq \nabla^2 F(Y)[U,U] \leq L_F \|U\|^2 \,\implies\, 0 \leq \nabla^2 G(Y)[U,U] \leq (L_F-\mu_F) \|U\|^2. \]
\qed
\end{proof}} 
\subsection{Relative Smoothness and the Bregman Iteration Map} 
In this section, we recall the framework of \cite{Bauschke2017,Bolte2018} to derive non-Euclidean gradient methods.

The first essential step is the choice of a \textit{distance kernel}. In our context, we choose a differentiable strictly convex function $h:\reals^{n \times r} \rightarrow \reals$, with ${\dom h = \reals^{n \times r}}$ (although more general distance kernels can be used). The distance kernel $h$ induces in turn a \textit{Bregman distance}
\BEQ
    D_h(X,Y) = h(X) - h(Y) - \la \nabla h(Y), X-Y \ra.
\EEQ
Note that $D_h$ is  not a proper distance, it is sometimes referred to as a \textit{Bregman divergence}. However $D_h$ enjoys a distance-like separation property: $D_h(X,X) = 0$ and $D_h(X,Y) > 0$ for $X \neq Y$. The choice of a distance kernel suited to the function $f$ is guided by the following relative smoothness condition, also called generalized Lipschitz property.

\begin{definition}[Relative smoothness \cite{Bauschke2017}]
We say that a differentiable function $f : \reals^{n \times r} \rightarrow \reals$ is $L$-smooth relatively to the distance kernel $h$ if there exists $L > 0$ such that for every $X,Y \in \rnr$,
\BEQ\label{eq:relative_smoothness}\tag{RelSmooth}
    f(X) \leq f(Y) + \la \nabla f(Y), X-Y \ra + L\, D_h(X,Y).
\EEQ
\end{definition}
For twice differentiable functions, relative smoothness has an elementary characterization:  $f$ is $L$-smooth relatively to $h$ if and only if 
\BEQ\label{eq:hessian_cond}
     \nabla^2 f(X)[U,U] \leq L\, \nabla^2 h(X)[U,U], \quad \forall X,U \in \reals^{n \times r}
\EEQ
{where $\nabla^2 f(X)[U,U]$ denotes the second derivative of $f$ at $X$ in the direction $U$.}
Notice that if $h(X) = \frac{1}{2} \|X\|^2$, then $D_h(X,Y) = \frac{1}{2}\|X-Y\|^2$ and we recover the standard Euclidean descent lemma that would be implied by Lipschitz continuity of the gradient of $f$. 

\paragraph{Bregman iteration map} Now that we are equipped with a non-Euclidean geo\-me\-try generated by $h$, we define the Bregman proximal iteration map with step size $\lambda$ as follows.
\BEQ\label{eq:bregman_prox}
T_\lambda(X) = \argmin_{U \in \reals^{n\times r}} \left\{ g(U) + f(X) + \la \nabla f(X), U-X \ra + \frac{1}{\lambda} D_h(U,X) \right\},
\EEQ
which consists in minimizing a surrogate for $\Psi$ where $f$ has been replaced by the upper approximation given by~\eqref{eq:relative_smoothness} and the nonsmooth part $g$ is kept intact, generalizing thus the approach used in the proximal gradient method.
The relative smoothness condition ensures that this operation decreases the objective $\Psi$ when $\lambda \in ]0,1/L]$.
This iteration map is the basic brick for non-Euclidean methods \`a la Bregman.  The simplest method is NoLips \cite{Bauschke2017} and its extension Dyn-NoLips (Algorithm \ref{algo:nolips}), which simply amounts to iterating $X^{k+1} = T_{\lambda_k}(X^k)$, but other possibilities exist using momentum ideas \cite{Auslender2006,Hanzely2018,Mukkamala2019}.

\subsection{The Quartic Geometry}\label{ss:quartic}
In order to provide some insight into the quartic geometry of our problem, let us consider the example where $F$ is a \textit{quadratic} function, i.e.,
\BEQ\label{eq:quad_f}
F(Y) = \frac{1}{2} \la \mathcal{A}Y, Y\ra + \la B,Y \ra \quad \forall \,Y \in \rnn,
\EEQ
where $B \in \rnn$ and $\mathcal{A}:\rnn \rightarrow \rnn$ is some linear map. Then, $f$ writes
\[f(X) = F(XX^T) = \frac{1}{2}\la \mathcal{A}(XX^T) , XX^T\ra + \la BX,X\ra \quad \forall\, X \in \rnr.\]
Clearly, $f$ is a \textit{quartic} function and its gradient is not Lipschitz continuous on $\rnr$, as the Hessian ``grows" to infinity when $\|X\|\rightarrow \infty$. In other words, \eqref{eq:relative_smoothness} does not hold with the Euclidean kernel $h = \frac{1}{2}\|\cdot\|^2$. We now show that relative smoothness holds with a family of well-chosen quartic kernels, which are more adapted to the geometry of $f$.

\subsubsection{The Quartic Norm Kernel}\label{ss:kernels}
We begin with the simplest quartic kernel, which depends solely on the Frobenius norm of $X$. Define the \textit{norm kernel} $h_N$ as
\begin{equation}\label{eq:norm_kernel}
    h_N(X) =  \frac{\alpha}{4}\|X\|^4 + \frac{\sigma}{2}\|X\|^2 \quad \forall \, X \in \rnr,
\end{equation}
where $\alpha,\sigma > 0$ are fixed parameters. Note that this kernel is not new by itself, as it has been already studied in \cite{Bolte2018} for vectors in $\reals^n$. Our first contribution is to show that it is adapted to every function of our class of problems. 

\begin{proposition}[Norm kernel]\label{prop:rel_smoothness}
The function $f$ is $1$-smooth relative to the norm kernel $h_N$ for $\alpha \geq 6L_F$ and $\sigma \geq 2\|\nabla F(0)\|$.
\end{proposition} 
\begin{proof}
 As $F$ is twice differentiable, then so if $f$ and we can use the Hessian characterization \eqref{eq:hessian_cond}. For $X, U \in \reals^{n \times r}$, the second derivative of $h_N$ is written
\begin{equation}\label{eq:hess_h}
\begin{split}
    \nabla^2 h_N(X)[U,U] &= \alpha \left(\|X\|^2 \|U\|^2 + 2 \la X, U \ra^2 \right)  + \sigma \|U\|^2 \\
         &\geq \alpha \|X\|^2 \|U\|^2 + \sigma \|U\|^2.
\end{split}
\end{equation}
On the other hand, the second derivative of $f$ is 
\BEQ \label{eq:expr_hess_f}
    \nabla^2 f(X)[U,U] = \nabla^2 F(XX^T)[UX^T + XU^T, UX^T + XU^T] + 2 \la \nabla F(XX^T), UU^T \ra.
\EEQ
Since $F$ has a Lipschitz continuous gradient, the standard second derivative inequality yields 
\[
\nabla^2 F(XX^T)[UX^T + XU^T, UX^T + XU^T] \leq L_F \|UX^T + XU^T\|^2.
\]
Now, the second term can be bounded by using the triangle inequality, the Cauchy-Schwarz inequality and the gradient Lipschitz property, to get
\[
    \begin{split}
      \la \nabla F(XX^T), UU^T \ra &= \la  \nabla F(0), UU^T \ra + \la \nabla F(XX^T) - \nabla F(0), UU^T \ra\\
        &\leq \|\nabla F(0)\| \, \|U\|^2 + \|\nabla F(XX^T) - \nabla F(0)\| \, \|UU^T\| \\
        &\leq \Big( \|\nabla F(0)\| + L_F\|XX^T\| \Big) \|U\|^2
    \end{split}
\]
hence
\begin{equation}\label{eq:hess_f_maj}
    \begin{split}
        \nabla^2 f(X)[U,U] &\leq L_F \|UX^T + XU^T\|^2 + 2 \left(L_F\|XX^T\| + \|\nabla F(0)\| \right) \|U\|^2 \\
        &\leq 2L_F \left( \|UX^T\|^2 + \|XU^T\|^2 \right) + 2 \left(L_F\|XX^T\| + \|\nabla F(0)\| \right) \|U\|^2 \\
        &\leq 6L_F \|X\|^2 \|U\|^2 + 2 \|\nabla F(0)\| \|U\|^2 \\
        &\leq \alpha \|X\|^2 \|U\|^2 + \sigma \|U\|^2
    \end{split}
\end{equation}   
where we used the submultiplicative property of the Frobenius norm, and our choice of parameters $\alpha,\sigma$. Combining \eqref{eq:hess_h} and \eqref{eq:hess_f_maj} gives that \[{\nabla^2 f(X)[U,U] \leq \nabla^2 h_N(X)[U,U]}\] for all $X, U \in \reals^{n\times r}$, hence that $f$ is $1$-smooth relatively to $h$ \cite{Bauschke2017}.\qed
\end{proof}
The Bregman iteration map \eqref{eq:bregman_prox} associated with the kernel $h_N$ can be computed easily in closed form. We give its expression in the unconstrained case~\cite{Bolte2018}. 
\begin{proposition}[Bregman iteration map for $h_N$, unconstrained case]
    Assume that there is no penalty term, i.e., that $g \equiv 0$. The Bregman iteration map of the norm kernel $h_N$ with step size $\lambda > 0$ is given by
    \[
         T_\lambda(X) = \frac{1}{\tau_\sigma(\alpha \|U\|^2)} U
    \] where \[
           U = \nabla h_N(X) - \lambda \nabla f(X) = (\alpha \|X\|^2 + \sigma)X - \lambda \nabla f(X)
    \]
    and $\tau_\sigma(c)$ denotes the unique real solution $z$ to the cubic equation $z^2(z-\sigma) = c$.
\label{prop:mirror_norm}
\end{proposition}

Note that $\tau_\sigma(c)$ can be computed in closed form using Cardano's method
\BEQ \label{eq:cardano}
\tau_\sigma(c) = \frac{\sigma}{3} + \sqrt[3]{\frac{c + \sqrt{\Delta}}{2} + \frac{\sigma^3}{27}}
    + \sqrt[3]{\frac{c - \sqrt{\Delta}}{2} + \frac{\sigma^3}{27}}
    \,\,\text{where}\,\,
    \Delta =c^2 + \frac{4}{27} c \sigma^3.
\EEQ
Compared to a standard gradient iteration, the additional operations are elementary and have a minimal impact on the arithmetic complexity.


\paragraph{Constraints and regularization terms.} Following the ideas in \cite{Bolte2018}, the Bregman iteration map of $h_N$ can also be easily computed in closed form when $g$ is the $\ell_1$ norm or the $\ell_0$ pseudonorm. As we will show in Section~\ref{ss:SymNMF}, this is also elementary when $g$ is the indicator function of the nonnegative orthant.

\subsubsection{A More Refined Kernel for Unregularized Problems: the Gram Kernel}\label{ss:gram}
While the kernel $h_N$ is simple and compatible with many penalties $g$, a better kernel can be derived for unconstrained instances by considering a richer geometry involving the Gram matrix. Define the \textit{Gram kernel} as
\begin{equation}\label{eq:gram_kernel}
    h_G(X) = \frac{\alpha}{4}\|X\|^4 + \frac{\beta}{4}\|X^T X\|^2 + \frac{\sigma}{2}\|X\|^2 \quad \forall \, X \in \rnr,
\end{equation}
where $\alpha, \beta \geq 0, \,\sigma > 0$ are given parameters. The Gram kernel is more refined than the previous norm kernel since it incorporates some nonisotropic information with the $\|X^TX\|^2$ term. {To show where this term stems from, observe that following Lemma \ref{lemma:muf}, $F$ can be decomposed as $F = \frac{\mu_F}{2}\|\cdot\|^2 + \tilde{F}$ where $\tilde{F}$ is $(L_F-\mu_F)$-smooth. Hence $f$ writes
\BEQ\label{eq:f_decomp}
f(X) = F(XX^T) = \frac{\mu_F}{2} \|X X^T\|^2 + \tilde{F}(XX^T).
\EEQ
Since $\|XX^T\|^2 = \|X^TX\|^2$, the first term can be directly incorporated into the kernel, which allows to prove a tighter relative smoothness inequality.}
\begin{proposition}[Gram kernel]\label{prop:rel_smoothness_gram}  $f$ is $1$-smooth relatively to the Gram kernel $h_G$ {for $\alpha \geq 2(L_F-\mu_F)$, $\beta \geq 2 L_F$ and $\sigma \geq 2\|\nabla F(0)\|$.}
\end{proposition} 

\begin{proof} 
This amounts to refine the analysis of the proof of Proposition~\ref{prop:rel_smoothness}. Let $X,U \in \reals^{n\times r}$. The second derivative of $h_G$ at $X$ in the direction $U$ writes
\begin{equation}
    \begin{split}
        \nabla^2 h_G(X)[U,U] &= \alpha \left(\|X\|^2 \|U\|^2 + 2 \la X, U \ra^2 \right)\\
        &\quad\quad+ \beta \left( \frac{1}{2}\|UX^T + XU^T\|^2 + \|U^T X\|^2 \right)  + \sigma \|U\|^2 \\
        &\geq \alpha \|X\|^2 \|U\|^2 + \beta \left( \frac{1}{2} \|UX^T + XU^T\|^2 + \|U^T X\|^2  \right) + \sigma \|U\|^2.
    \end{split}
\end{equation}
On the other hand, following~\eqref{eq:expr_hess_f} the second derivative of $f$ satisfies
\begin{equation*}
        \nabla^2 f(X)[U,U] \leq L_F \|UX^T + XU^T\|^2 + 2\la \nabla F(XX^T),UU^T\ra
\end{equation*}
To bound the second term, we use Lemma \ref{lemma:muf} which states that the function $G(Y):=F(Y)-\mu_F \|Y\|^2 / 2$ is convex and smooth with constant $L_F-\mu_F$. Using the gradient Lipschitz property of $G$ yields
\begin{align*}
    \la \nabla F(XX^T),UU^T\ra &= \la \nabla F(0), UU^T\ra + \mu_F \la XX^T, UU^T \ra \\
    &\qquad+ \la \nabla F(XX^T) - \nabla F(0) - \mu_F (XX^T - 0), UU^T  \ra \\
    &= \la \nabla F(0),UU^T \ra + \mu_F \|U^T X\|^2 + \la \nabla G(XX^T)-  \nabla G(0),UU^T \ra \\
    &\leq \|\nabla F(0)\| \,\|U\|^2 + \mu_F \|U^T X\|^2 + (L_F - \mu_F) \|XX^T\|\, \|UU^T\|\\
    &\leq \|\nabla F(0)\| \,\|U\|^2 + L_F \|U^T X\|^2 + (L_F - \mu_F) \|X\|^2\|U\|^2,
 \end{align*}
using that $\mu_F \leq L_F$, and so we have
\begin{align*}
    \nabla^2 f(X)[U,U] &\leq 2(L_F - \mu_F) \|X\|^2 \|U^2\|+ L_F \|UX^T + XU^T\|^2 + 2 L_F \|U^T X\|^2\\
    &\qquad+ 2 \|\nabla F(0)\|\, \|U\|^2\\
    &\leq  \alpha \|X\|^2 \|U\|^2 + \frac{\beta}{2} \|UX^T + XU^T\|^2 + \beta \|U^T X\|^2   + \sigma \|U\|^2\\
    &\leq \nabla^2 h_G(X)[U,U]
\end{align*}
which shows that, for the prescribed choice of $\alpha,\beta,\sigma$, the function $f$ is 1-smooth relatively to $h_G$.
\qed
\end{proof}
{
\paragraph{Approximation quality for well-conditioned $F$.} Let us illustrate the advantage of the Gram kernel when $F$ is well-conditioned. For simplicity, assume here that $F$ is a quadratic function, as in \eqref{eq:quad_f}, i.e., $F(Y) = \frac{1}{2}\la \mathcal{A}Y,Y \ra + \la B, Y \ra$ where $\mathcal{A}$ is a positive semidefinite linear operator on $\reals^{n \times r}$, and hence $f$ has a \textit{quartic} and a \textit{quadratic} term
\[f(X) = \frac{1}{2}\la \mathcal{A}(XX^T),XX^T \ra + \la B, XX^T \ra.
\]
The gap between $f$ and $h_G$ with the choice of coefficients prescribed by Proposition \ref{prop:rel_smoothness_gram} writes, for $X \in \rnr$,
\BEQ
\begin{split}
    h_G(X) - f(X) &= \frac{(L_F - \mu_F)}{2}\|X\|^4 + \frac{L_F}{2}\|X^TX\|^2 + \|\nabla F(0)\|\,\|X\|^2 \\
    &\qquad-  \frac{1}{2}\la \mathcal{A}(XX^T),XX^T \ra - \la B X, X \ra \\
    &= \underbrace{\frac{(L_F - \mu_F)}{2}\|X\|^4 + \frac{1}{2} \la (L_F I - \mathcal{A})(XX^T), XX^T\ra}_{d_4(X)} \\
    &\qquad + \underbrace{\la (\|\nabla F(0)\| I-B)X, X\ra}_{d_2(X)}
\end{split}
\EEQ
where we separated the gap into a quartic term $d_4$ and a quadratic term $d_2$. It can be seen from \eqref{eq:hessian_cond} that the quality of approximation of the kernel is given by the difference of the Hessians. Focusing on the quartic part, the Hessian difference is
\begin{align*}
    \nabla^2 d_4(X)[U,U] &= 2(L_F - \mu_F)\left( \|X\|^2 \|U\|^2 + 2 \la X,U\ra^2 \right) +  2 \la (L_F I - \mathcal{A})(XX^T),UU^T \ra  \\
    &\qquad + \la (L_F I - \mathcal{A})(UX^T + XU^T),UX^T + XU^T\ra\\
    &\leq 6(L_F - \mu_F)\|X\|^2 \|U\|^2  \\
    &\qquad + \| L_F I - \mathcal{A} \|_{\text{op}}\left( 2 \|XX^T\|\, \|UU^T\| + \|UX^T + XU^T\|^2\right)
\end{align*}
for $X, U \in \reals^{n\times r}$. 
Recalling that $F$ is $L_F$-smooth and $\mu_F$-strongly convex, we have that ${\|L_FI - \mathcal{A}\|_{\text{op}}}\leq (L_F - \mu_F)$, therefore
\begin{equation*}
\begin{split}
    \nabla^2 d_4(X)[U,U] &\leq (L_F - \mu_F)\left( 6\|X\|^2 \|U\|^2 + 2 \|XX^T\|\,\|UU^T\| + \|UX^T + XU^T\|^2 \right)\\
    &\leq 12 L_F (1 - \frac{\mu_F}{L_F}) \|X\|^2 \|U\|^2
\end{split}
\end{equation*}
which shows that the quality of approximation of the quartic part of $f$ by the Gram kernel depends on the condition number $\kappa_F := L_F /\mu_F$ of $F$. 
Note that one could actually refine the analysis by replacing $\kappa_F$ with the condition number of $F$ restricted to the set of matrices of rank at most $2r$, which can be much smaller.
This is the case when the linear map $\mathcal{A}$ satisfies the \textit{restricted isometry property} (RIP), which occurs with high probability in matrix sensing applications with a sufficiently large number $n$ of samples \cite{Recht2007,Meka2009,Jain2012}.
 }

\paragraph{Computing the iteration map.} We show now that, when there is no penalty term $g$, the Bregman iteration map of $h_G$ can be computed efficiently, as it involves solving an easy quartic minimization subproblem of size $r$.
\begin{proposition}[Gram's iteration map]
\label{prop:gram_gradient_map} Assume that $g \equiv 0$.
For  $X \in \reals^{n \times r}$, the Bregman iteration map of $f$ for the Gram kernel $h_G$ with step size $\lambda > 0$, called Gram's iteration map, is given by
    \begin{equation*}
        T_\lambda(X) = V \left[\alpha \Tr(Z) I_r + \beta Z + \sigma I_r \right]^{-1}  
    \end{equation*}
 where the matrices $V,Z$ are computed through the routine:\begin{itemize}
\item  Set  $V = \nabla h_G(X) - \lambda \nabla f(X)$,
\item  diagonalize $V^TV$ as $V^T V = P^T D P$ where $P \in \mathcal{O}_r$ and $D = \diag(\eta_1^2,\dots,\eta_r^2)$, 
\item let $\mu = (\mu_1, \dots,\mu_r)$ be the unique solution of the convex minimization problem
    \begin{equation*}
        \min_{x \in \reals^r} \phi(x) := \frac{\alpha}{4}\|x\|^4 + \frac{\beta}{4}\sum_{i=1}^r x_i^4 + \frac{\sigma}{2}\|x\|^2 - \sum_{i=1}^r \eta_i x_i,
    \end{equation*}
\item finally set $Z = P^T \diag\big[ \mu_1^2,\dots, \mu_r^2 \big] P.$
\end{itemize}
\end{proposition}

\begin{proof}
When $g \equiv 0$, The Bregman iteration map of $h_G$ writes, for $X \in \reals^{n\times r}$,
\begin{equation}\label{eq:breg_map_gram}
    \begin{split}
        T_\lambda(X) &= \argmin_{U \in \reals^{n\times r}} \left\{\la \nabla f(X), U-X \ra + \frac{1}{\lambda} D_{h_G}(U,X) \right\} \\
        &= \argmin_{U \in \reals^{n\times r}} \left\{ h_G(U) - \la V,U \ra \right\} \\
    \end{split}
\end{equation}
where we remove constant terms and defined $V := \nabla h_G(X) - \lambda \nabla f(X)$.
Write for the sake of clarity $U^\star := T_\lambda(X)$. The optimization problem \eqref{eq:breg_map_gram} is strictly convex and the unique solution $U^\star$ satisfies $\nabla h_G(U^\star) = V$, meaning that
\begin{equation}\label{eq:Z_cond}
    U^\star \left(\alpha \|U^\star\|^2 I_r + \beta U^{\star T} U^\star + \sigma I_r \right) = V.
\end{equation}
Define $Z := U^{\star T} U^\star \in \reals^{r \times r}$. Then, the knowledge of $Z$ determines $U^\star$, since $\|U^\star\|^2 = \Tr(Z)$ and therefore $U^\star = V (\alpha \Tr(Z) I_r + \beta Z + \sigma I_r)^{-1} $.

Now, taking \eqref{eq:Z_cond} and multiplying by its transpose implies that 
\begin{equation}\label{eq:z_v_eq}
    (\alpha \|U^\star\|^2 I_r  + \beta Z + \sigma I_r)^2 Z = V^T V.
\end{equation}
This shows that $V^T V$ is a polynomial in $Z$, and therefore that they admit the same eigenvectors. Write the diagonalization $V^T V = P^T \diag(\eta_1^2,\dots,\eta_r^2) P$ and $Z = P^T \diag(\mu_1^2,\dots,\mu_r^2) P$ where $P \in \mathcal{O}_r$ and $\mu_i,\eta_i \geq 0$ for $i = 1\dots r$. It follows from diagonalizing \eqref{eq:z_v_eq} and taking the square root that
\begin{equation}
    \left( \alpha \left(\sum_{j=1}^r \mu_j^2 \right)  + \beta \mu_i^2 + \sigma \right) \mu_i = \eta_i \quad \forall i = 1,\dots,r 
\end{equation}
This is exactly the first-order optimality condition on $\mu = (\mu_1,\dots,\mu_r)$ for the problem
\begin{equation}\label{eq:gram_subprob_reminder}
    \mu = \argmin_{x \in \reals^r} \frac{\alpha}{4}\|x\|^4 + \frac{\beta}{4}\sum_{i=1}^r x_i^4 + \frac{\sigma}{2}\|x\|^2 - \sum_{i=1}^r \eta_i x_i.
\end{equation}
Note that we do not need to enforce the nonnegativity constraint on $x$, since we chose $\eta_i \geq 0$ it follows that the optimal solution will be nonnegative. Hence, we can reconstruct $Z$ from the diagonalization of $V^T V$ and the solution of Problem \eqref{eq:gram_subprob_reminder}, and thus we get the procedure described in the theorem for computing $U^* = T_\lambda(X)$. \qed
\end{proof}

\paragraph{Complexity} Note that the order of multiplication is important: we only need to compute the eigendecomposition of $V^T V$, which is of size $r \times r$. We additionally need to solve a small minimization problem of size $r$, which can be done efficiently using the quartic NoLips algorithm with norm kernel (see Appendix \ref{ss:gram_kernel_algo} for implementation details). Due to this, the complexity of computing the Bregman iteration map of $h_G$ is $O(nr^2 + r^3 + K r)$, where $K$ is the number of iterations needed to solve the subproblem.
Since $r$ is usually much smaller than $n$ by several orders of magnitude, the main computational bottleneck remains in most applications computing the gradient $\nabla f(X)$.

{\subsubsection{Comparison: how to choose the most appropriate kernel}
In order to devise efficient methods, one should search for the kernel $h$ such that the upper approximation of $f$ in~\eqref{eq:relative_smoothness} is \emph{as tight as possible}, or, equivalently, such that the Hessian of the residual $Lh-f$ is small. On the other hand, $h$ has to be simple enough so that the iteration map \eqref{eq:bregman_prox} is \emph{easy to compute} (which precludes choosing $h=f$, as the iteration would be as hard to solve as the initial problem). This trade-off is key in choosing the appropriate kernel. Let us review these two conflicting criteria in our situation.

\paragraph{Complexity of the Bregman iteration map.} For the norm kernel $h_N$, one iteration involves computing the gradient of $f$, then solving a simple scalar equation. The Gram kernel $h_G$ involves solving a subproblem which requires $O(nr^2 + r^3)$ additional operations. This overhead is negligible for the typical regime where $r \ll n$; however, the iterate can be computed easily only for unconstrained problems.

\paragraph{Quality of Hessian approximation.} We showed in Section \ref{ss:gram} that the quality of the approximation of the quartic component of $f$ by the Gram kernel is bounded by $O(1-\mu_F/L_F)$. Therefore, it is expected to show good performance when $F$ is sufficiently well-conditioned. The norm kernel, however, has no such property, as its approximation of $f$ is much coarser. The difference stems from the supplementary $\|X^T X\|^2$ term, which can be much smaller than $\|X\|^4$, especially when the columns of $X$ are nearly orthogonal.

Note that even if $F$ is not globally strongly convex or $\mu_F$ is unknown, the Gram kernel can take advantage of local strong convexity through adaptive step sizes, as we show in the sequel.

}

\section{Algorithms for Quartic Low-Rank Minimization}

Now that we are equipped with a non-Euclidean geometry induced by one of the kernels $h_N$ and $h_G$, we are ready to define the minimization scheme Dyn-NoLips in Algorithm \ref{algo:nolips}. It extends the NoLips algorithm from \cite{Bolte2018} to allow step sizes larger than the theoretical value $1/L$.

\begin{algorithm}[H]
	\begin{algorithmic}
		\STATE {\bfseries Input:} A distance kernel $h$ such that $f$ is smooth relatively to $h$ and a maximal step size $\lambda_{\rm max}$
		\STATE Initialize $X^0 \in \reals^{n\times r}$ such that $\Psi(X^0) < \infty$.
		\FOR{k = 1,2,\dots}
		\STATE Choose a step size $\lambda_k \leq \lambda_{\rm max}$ such that the sufficient decrease condition \eqref{eq:suff_decrease} holds
		\STATE 
		    Set $X^k = T_{\lambda_k}(X^{k-1}) $
		\ENDFOR
	\end{algorithmic}
	\caption{Dyn-NoLips}
	\label{algo:nolips} 
\end{algorithm}

\paragraph{Step size choice} The step size $\lambda_k$ is chosen so that the new iterate $X^k = T_{\lambda_k}(X^{k-1})$ satisfies
\begin{equation}\label{eq:suff_decrease}
f(X^k) \leq f(X^{k-1}) + \la \nabla f(X^{k-1}), X^k - X^{k-1} \ra + \frac{1}{\lambda_k} D_h(X^k, X^{k-1}).
\end{equation}
There are two ways to ensure this condition holds.
\begin{itemize}
    \item \textbf{Fixed step size.} Since $f$ is $L$-smooth relatively to $h$, \eqref{eq:suff_decrease} holds as soon as $0 < \lambda_k \leq {1}/{L}$.
  \item \textbf{Dynamical step size.} In some cases, the relative Lipschitz constant might be too conservative, and better numerical performance can be achieved by taking larger steps. We therefore can use a dynamical strategy for extending the step size, ensuring that \eqref{eq:suff_decrease} holds at each iteration. There are many strategies to efficiently adjust the step size; see, e.g., \cite{Nesterov2007}. {In our case, we choose a simple strategy similar in spirit to the Armijo line search: at iteration $k$, start with a tentative step size $\lambda_k$, then find the smallest integer $j$ such that \eqref{eq:suff_decrease} is satisfied with step size $2^{-j}\lambda_k$.
  Then, set $\lambda_{k+1} = 2^{-j+1} \lambda_k$. }
\end{itemize}

\paragraph{Convergence to a stationary point.} We now extend the theoretical convergence results from \cite{Bolte2018} to handle the dynamical step size strategy.

\begin{theorem}[Convergence results]\label{thm:nolips}
Let $\{X^k\}_{k \geq 0}$ be the sequence generated by Algorithm \ref{algo:nolips}. Assume that 
    \begin{enumerate}
        \item $f$ is $L$-smooth relatively to a distance kernel $h$ such that $h$ is strongly convex and twice continuously differentiable on $\reals^{n\times r}$, and the penalty function $g$ is convex.
        \item The function $\Psi = f+g$ is coercive (meaning that $\Psi(X)\rightarrow + \infty$ when $\|X\|\rightarrow +\infty$) and semialgebraic.
    \end{enumerate}
Then, the  sequence $\{\Psi(X^k)\}_{k \geq 0}$ is nonincreasing, and the sequence $\{X^k\}_{k \geq 0}$ converges towards a critical point $X^*$ of problem~\eqref{eq:nolips_min_problem}.
\end{theorem}

\begin{proof}
First, the step size $\lambda_k$ can be bounded for $k\geq 0$ as
    \begin{equation}\label{eq:stepsize_bound}
        \frac{1}{2L} \leq \lambda_k \leq \lambda_{\rm max}.
    \end{equation}
Indeed, the upper bound holds by construction of the algorithm. The lower bound comes from the relative smoothness property: condition \eqref{eq:suff_decrease} is true for every $\lambda \in (0, \frac{1}{L}]$, so the inner loop will stop whenever $\lambda$ gets below $1/L$.

    Let us now prove the result. Since Condition \eqref{eq:suff_decrease} holds at each iteration $k$, we can write
    \begin{equation}\label{eq:dec_f}
            f(X^{k+1}) \leq f(X^k) + \la \nabla f(X^k), X^{k+1} - X^k \ra + \frac{1}{\lambda_k} D_h(X^{k+1}, X^k).
    \end{equation}
    On the other hand, the optimality condition characterizing $X^{k+1} = T_{\lambda_k}(X^k)$ writes
    \begin{equation}\label{eq:cond_optim}
        0 \in \lambda_k \left(\partial g(X^{k+1}) + \nabla f(X^k) \right) + \nabla h(X^{k+1}) - \nabla h(X^{k}),
    \end{equation}
    where $\partial g$ denotes the subdifferential of the convex function $g$. Combining \eqref{eq:cond_optim} with the subgradient inequality for $g$ yields
    \begin{multline}\label{eq:dec_g}
            g(X^{k+1})\leq g(X^k) +\frac{1}{\lambda_k} \la \nabla h(X^k) - \nabla h(X^{k+1}), X^{k+1} - X^k \ra \\
            - \la \nabla f(X^k), X^{k+1} - X^k \ra.
    \end{multline}
    Summing \eqref{eq:dec_f} and \eqref{eq:dec_g} gives
    \begin{equation*}\label{eq:descent_ineq}
        \Psi(X^{k+1}) \leq \Psi(X^k) + \frac{1}{\lambda_k} \big[ D_h(X^{k+1}, X^k)  + \la \nabla h(X^k) - \nabla h(X^{k+1}), X^{k+1} - X^k  \ra \big],
    \end{equation*}
    which yields 
    \begin{equation}\label{eq:desc_cond}
       \Psi(X^{k+1}) \leq \Psi(X^k) - \frac{1}{\lambda_k} D_h(X^{k}, X^{k+1}).
    \end{equation}
From this inequality, we can now prove the same convergence properties as for the standard NoLips scheme.
 Indeed, the monotonicity of the sequence $\{\Psi(X^k)\}_{k \geq 0}$ is a direct consequence of the above. Since $\lambda_k \leq \lambda_{\rm max}$, it follows that at every iteration $k \geq 0$,
\begin{equation*}
    \Psi(X^{k}) - \Psi(X^{k+1}) \geq \frac{1}{\lambda_{\rm max}} D_h(X^k, X^{k+1}).
\end{equation*}

    Now, this inequality is the same as the one needed to prove convergence in the case of the fixed step size in \cite{Bolte2018}. Thus, global convergence towards a critical point is a consequence of \cite[Th.~4.1]{Bolte2018}, since all the assumptions are met:  
    the kernel $h$ is defined over the entire space $\reals^{n \times r}$, it is strongly convex, and $\nabla h$ is Lipschitz continuous on bounded subsets of $\reals^{n \times r}$ (because we assumed it is $C^2$). We also need the fact that the sequence $\{X^k\}_{k \geq 0}$ is bounded, which is a consequence of the monotonicity of $\{\Psi(X^k)\}_{k \geq 0}$   and the fact that the function $\Psi$ is coercive.\qed
\end{proof}

The semialgebraicity assumption is needed to establish the crucial nonsmooth Lojasiewicz property \cite{Bolte2007}, required to show convergence to a critical point. It holds for all the applications we cited, since the class of semialgebraic functions includes polynomial functions, $\ell_1$ and $\ell_2$ norms, the $\ell_0$ seminorm and indicators of polynomial sets.

\section{Applications}\label{s:applications}
We now illustrate applications of our methodology to two different low-rank problems, symmetric nonnegative matrix factorization and Euclidean distance matrix completion.
We show that good numerical performance can be reached using the dynamical step strategy, and that, for Euclidean matrix completion, it can be further improved by using the Gram kernel.

\subsection{Symmetric Nonnegative Matrix Factorization}\label{ss:SymNMF}
Symmetric Nonnegative Matrix Factorization (SymNMF) is the task of finding, given a symmetric nonnegative matrix $M \in \reals^{n\times n}$, a nonnegative matrix $X\in \reals^{n\times r}$ such that 
$M \approx XX^T$. This is done by solving
\BEQ\label{eq:symnmf}\tag{SymNMF}
\BA{ll}
\mbox{min} & \frac{1}{2} \|M - X X^T\|_F^2\\    
\mbox{subject to} & X \geq 0
\EA\EEQ
in the variable $X\in\reals^{n \times r}$, where the inequality constraint is meant componentwise and $r \leq n$ is the target rank.

\eqref{eq:symnmf} is used as a probabilistic clustering or graph clustering technique \cite{Ding2005,He2011}. Numerical experiments by \cite{Kuang2015} have shown that it achieves state-of-the-art clustering accuracy on several text and image datasets.

\subsubsection{Solving SymNMF.} While \eqref{eq:symnmf} looks similar to the well-known asymmetric NMF problem $\min_{X,Y} \frac{1}{2}\|M - XY^T\|$, it is actually harder. This is because NMF has a favorable block structure that allows the application of efficient alternating algorithms \cite{Controller2013,Cichocki2009}. SymNMF, however, does not enjoy the same block structure. Current solvers fall into two categories:

{\em Direct solvers.} There have been several attempts at solving the original problem, including multiplicative update rules \cite{He2011}, projected gradient algorithm quasi-Newton schemes \cite{Kuang2015}, and coordinate descent \cite{Vandaele2016}.
    
{\em Nonsymmetric relaxations.} Another  idea is  to use a mere penalty method \cite{Kuang2015,Lu2017,zhu2018a}, relaxing \eqref{eq:symnmf} to the following penalized nonsymmetric problem
\BEQ\label{eq:pnmf}\tag{P-NMF}
\BA{ll}
\mbox{minimize} & \frac{1}{2} \|M - X Y^T\|_F^2 + {\mu} \|X - Y\|_F^2\\
\mbox{subject to} & X,Y \geq 0,
\EA
\EEQ
in the variables $X, Y \in \reals^{n \times r}$, with parameter $\mu \geq 0$. This formulation is very similar to asymmetric NMF and can be solved by the same fast alternating algorithms that exploit the block structure, such as Alternating Nonnegative Least Squares (ANLS) and Hierarchical Alternating Least Squares \cite{zhu2018a} (HALS), which are arguably the fastest SymNMF solvers.


\paragraph{Applying NoLips} We propose to apply NoLips for optimizing the original objective function.
Problem \eqref{eq:symnmf} falls within our framework with ${F(Y) = \frac{1}{2}\|M-Y\|^2}$, which has a Lipschitz gradient with constant $1$, and $g(X) = i_{\{X \geq 0\}}$ the indicator function of the nonnegative orthant. Therefore, Proposition \ref{prop:rel_smoothness} implies that $f(X) :=  \frac{1}{2}\|M-XX^T\|^2$ is $1$-smooth relatively to the kernel $h_N$ with $\alpha = 6$ and $\sigma = 2\|\nabla F(0)\| = 2\|M\|$. 
Since, in addition, $f$ is polynomial and $g$ is the indicator of a polynomial set, $f+g$ is semialgebraic, and it is also coercive, so Theorem \ref{thm:nolips} guarantees that NoLips will converge towards a stationary point of problem \eqref{eq:symnmf}.

In this problem, the Bregman iteration map is solved by simply adding a projection step
\[
 T_\lambda(X) = \frac{1}{\tau_\sigma(\alpha \|\Pi_+(U)\|^2)}\Pi_+(U),
\]
where $U = \nabla h_N(X) - \lambda \nabla f(X)$, $\tau_\sigma$ has been defined in Proposition \ref{prop:mirror_norm} and $\Pi_+$ is the projection on the nonnegative orthant: $\Pi_+(U) = \max(U, 0)$ (entrywise).

{\paragraph{Computational complexity for NoLips.} The computational complexity of an iteration is dominated by gradient computations and objective function evaluations, as all other operations are linear in the size of the variable.

If $M$ is a $n \times n$ \textbf{dense} matrix, each gradient and function evaluation uses $O(n^2 r + n r^2)$ floating point operations. If $M$ is represented as a \textbf{sparse} matrix with $p << n^2$ nonzero elements, then we can take advantage of this structure \cite[Rmk. 2]{Vandaele2016} by using
\begin{equation}
    \begin{split}
        f(X) &= \frac{1}{2}\|XX^T-M\|^2 = \frac{1}{2} \|M\|^2 + \frac{1}{2}\|X^T X\|^2 - \la MX, X \ra \\
        \nabla f(X) &= 2 X (X^T X) - 2 MX
    \end{split}
\end{equation}
which yields a much improved $O\left( (r^2 + p) n \right) $ complexity per iteration.}

\subsubsection{Numerical experiments}
We implemented the following algorithms: Algorithm~\ref{algo:nolips} with dynamical step size and the norm kernel (\texttt{Dyn-NoLips}), the  \texttt{$\beta$-SNMF} scheme from \cite{He2011}, where we set $\beta = 0.99$ as advised by the authors, the projected gradient algorithm (\texttt{PG}) with Armijo line search from  \cite{Kuang2015}, where we use the line search parameters $\beta = 0.1$ and $\sigma = 0.01$, the coordinate descent scheme (\texttt{CD}) from \cite{Vandaele2016}, {the \texttt{ADMM} algorithm \cite{Lu2017}}, and the two fast algorithms from \cite{zhu2018a} for solving the penalized problem \eqref{eq:pnmf}: \texttt{SymANLS} and \texttt{SymHALS}. For the last two, we tuned the $\mu$ penalization parameter for best performance.  We left out the quasi-Newton algorithm from \cite{Kuang2015} because of its prohibitive $O(n^3)$ complexity for large datasets.

All algorithms were implemented in Julia \cite{Bezanson2017} which is a highly-optimized numerical computing language. Since our algorithms have different complexity per iteration, it is essential to compare them in terms of running time, and Julia provides a fairly accurate way to do so as there is little interpreter overhead in loops.\footnote{Tests were run on a PC Intel CORE i7-4910MQ CPU @ 2.90 GHz x 8 with 32 Go RAM.}

We used two image and two text datasets.
\begin{itemize} \itemsep 0ex
    \item \textbf{Image.} 
    \begin{itemize}
        \item \textbf{CBCL}\footnote{http://cbcl.mit.edu/software-da\-tasets/Face\-Data2.html}: 2,429 images of faces of size $19 \times 19$ 
        
        \item \textbf{Coil-20}\footnote{http://www.cs.columbia.edu/CAVE/software/softlib/\-coil-20.php}: 1440 images of size $128 \times 128$ representing 20 objects under various angles.
    \end{itemize}

    \item \textbf{Text.}
        \begin{itemize}
            \item \textbf{TDT2}\footnote{http://www.cad.zju.edu.cn/home/dengcai/Data/TextData.html \label{footnote:tdt2}}: dataset of 11,201 news articles classified in 96 semantic categories. We used the version provided by Cai et al. \cite{CWH09,CMHZ08,CHZH07,CHH05}, which has been restricted to the largest 30 categories, leaving a total of 9,394 documents.
            \item \textbf{Reuters}\textsuperscript{\ref{footnote:tdt2}}: dataset of news articles, which we restricted to the largest 25 categories, leaving a total of 7,963 documents. 
        \end{itemize}
\end{itemize}

For all image and text datasets, we construct a sparse similarity matrix $M$ following the procedure described in \cite[Section 7.1]{Kuang2015}. We begin by computing the similarity graph between data points, using cosine similarity on term frequency vectors for text, and a Gaussian kernel for image (with the self-tuning method for the scale). The graph obtained is \textit{sparsified} by keeping only the edges connecting the k-nearest neighbors, with $k = \lfloor \log_2 n \rfloor + 1$. Then, $M$ is taken as a normalized version of the graph adjacency matrix.

{We use the usual convergence criterion for constrained nonconvex problems
 \begin{equation}\label{eq:stop_crit}
 \frac{\|\nabla^P f(X^k)\|}{\|\nabla^P f(X^0)\|}\leq \epsilon
 \end{equation}
 where $\nabla^P f(X)$ is the projected gradient defined as
 \[
 (\nabla^P f(X))_{ij} = \begin{cases}
\nabla f(X)_{ij}&\text{ if } X_{ij} > 0,\\
\min \left(\nabla f(X)_{ij}, 0\right)& \text{ if } X_{ij} = 0.
\end{cases}
 \]
 }
 
\begin{table}
{
    \centering
    
    \caption{CPU time (in seconds) needed to reach a decrease of $\epsilon = 10^{-3}$ in projected gradient norm (see \eqref{eq:stop_crit} for definition). Results have been averaged over 10 random initializations. Hyperparameters for SymHALS, SymANLS and ADMM have been tuned for best performance. Missing values indicate failure of convergence.}
    
    \begin{tabular}{|l|c|*{7}{>{\centering\arraybackslash}p{1cm}}|}
        \hline
        Dataset & r & {\scriptsize NoLips} & {\scriptsize PG} & {\scriptsize Beta} & {\scriptsize CD} & {\scriptsize SymHALS} & {\scriptsize SymANLS} & {\scriptsize ADMM} \\
        \hline
        \multirow{4}{*}{Coil-20}   & 10 &  24.7 &  51.4 & - &  26.2 &  7.0 &  32.3 &  -  \\
 & 20 &  23.7 &  36.8 & - &  21.3 &  4.0 &  18.2 & -  \\
 & 30 &  20.7 &  40.8 & - &  35.4 &  6.5 &  20.2 & -  \\
 & 40 &  21.7 &  49.5 & - &  57.6 &  7.5 &  28.4 &  -  \\
        \hline
        \multirow{4}{*}{CBCL} 
 
 & 10 &  38.2 &  42.7 &  44.0 &  35.6 &  13.6 &  35.2 &  42.8  \\
 & 20 &  57.7 &  88.4 &  - &  93.9 &  17.8 &  47.8 &  -  \\
 & 30 &  60.9 &  134.3 &  - &  135.0 &  15.1 &  43.4 &  -  \\
 & 40 &  50.8 &  126.4 &  - &  90.0 &  23.7 &  52.5 &  -  \\
        \hline
        \multirow{4}{*}{TDT2} 
        
 & 10 &  35.2 &  54.2 &  - &  97.5 &  11.0 &  - &  -  \\
 & 20 &  52.4 &  76.1 &  - &  109.9 &  20.1 &  - &  -  \\
 & 30 &  29.4 &  45.1 &  - &  - &  12.1 &  - &  -  \\
 & 40 &  28.0 &  49.8 &  - &  - &  17.7 &  - &  -  \\
        \hline
        \multirow{4}{*}{Reuters} 
        & 10 &  6.5 &  10.0 &  - &  33.0 &  3.0 &  54.2 &  -  \\
 & 20 &  28.7 &  32.8 &  - &  71.7 &  9.5 &  74.7 &  -  \\
 & 30 &  24.3 &  45.5 &  - &  69.4 &  6.5 &  91.0 &  -  \\
 & 40 &  40.2 &  68.5 &  - &  83.2 &  10.6 &  108.3 &  -  \\
        \hline
    \end{tabular}
    
    \label{tab:symnmf_runtime}
}
\end{table}

\begin{figure}

  \centering
  \subfigure[COIL-20 (image) $n = 1440$, $r = 20$]{
    \includegraphics[width = 0.43\columnwidth]{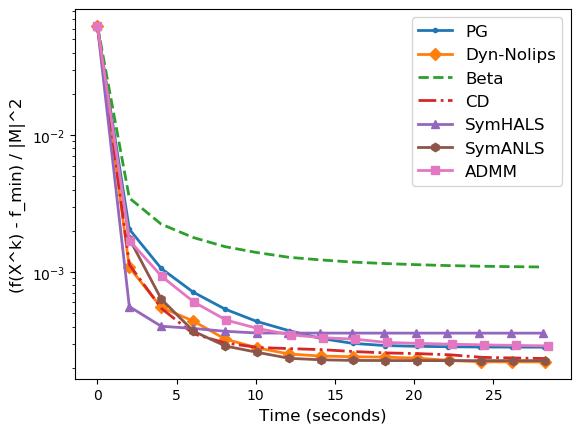}
  }
  \quad
  \subfigure[CBCL (image) $n = 2429$, $r = 20$]{
    \includegraphics[width = 0.43\columnwidth]{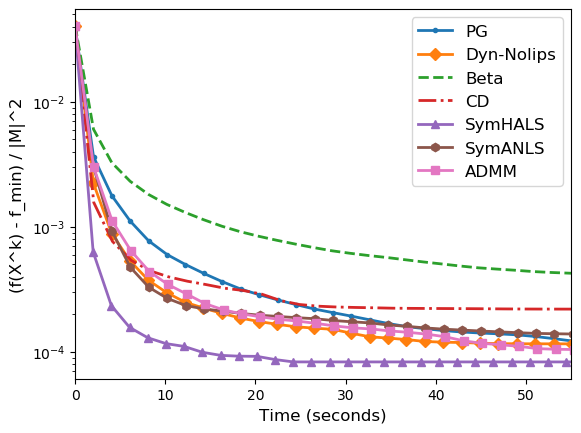}
}
  
  \centering
  \subfigure[TDT2 (text) $n = 9394$, $r = 30$]{
    \includegraphics[width = 0.43\columnwidth]{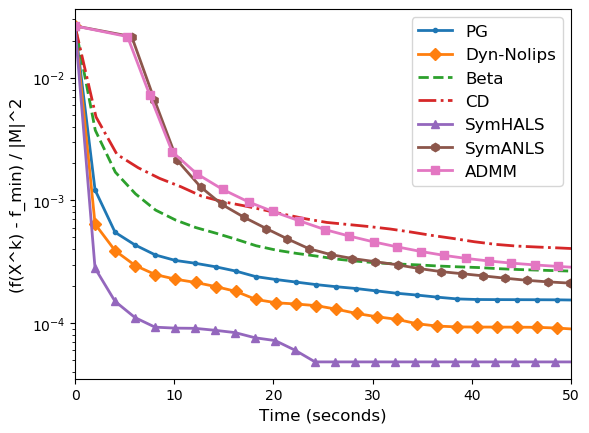}
    \label{fig:tdt2}
  }
  \quad
  \subfigure[Reuters (text) $n = 7963$, $r = 30$]{
    \includegraphics[width = 0.43\columnwidth]{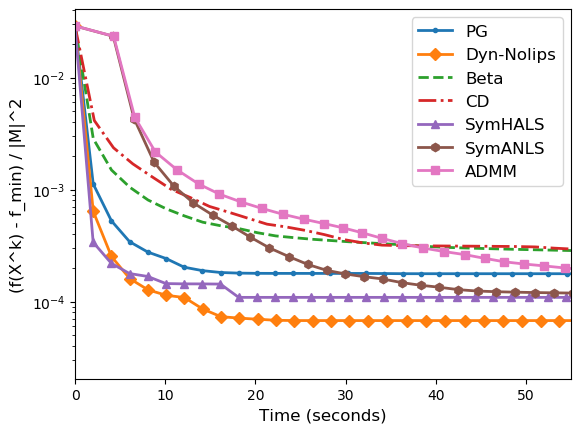}
    \label{fig:reuters}
  }

\caption{ {SymNMF normalized objective gap $\left(f(X^k) - f_{\rm min} \right) / \|M\|^2$ averaged over 10 random initializations, for various sparse similarity matrices $M \in \reals^{n \times n}$. Hyperparameters for SymHALS, SymANLS were tuned for best performance, while Dyn-NoLips is parameter-free.}}
\label{fig:symnmf}
\end{figure}

{Table \ref{tab:symnmf_runtime} reports the average time needed to reach a convergence criterion of $\epsilon = 10^{-3}$, for 10 random initializations. For each dataset, we test several values for the rank parameter $r$. In addition, Figure 1 shows the average evolution of the normalized objective gap $\left(f(X^k) - f_{min} \right) / \|M\|^2$, where $f_{min}$ is the minimal objective value encountered in all initializations.

}

{ Overall, the algorithm that shows the best convergence speed is \texttt{SymHALS}, but it has the disadvantage of needing to tune the penalization parameter $\mu$. In the experiments we report, small values of $\mu$ yielded optimal performance, while the convergence theory of \cite{Zhu2018b} only holds for large values for which the algorithm is much slower. By contrast, \texttt{Dyn-NoLips} is hyperparameter-free and has the second best overall performance. The gap with the other methods is particularly significant on the larger \texttt{TDT2} and \texttt{Reuters} datasets, showing that the method scales well with problem dimension. 


}

\begin{figure}
  \centering
  \subfigure[$n= 2000,r=3$]{
    \includegraphics[width = 0.4\columnwidth]{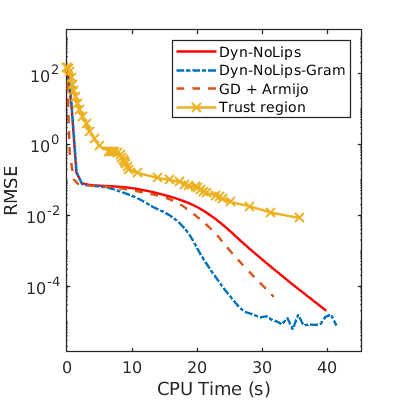}
  }
  \quad
  \subfigure[$n= 5000, r = 3$]{
    \includegraphics[width = 0.4\columnwidth]{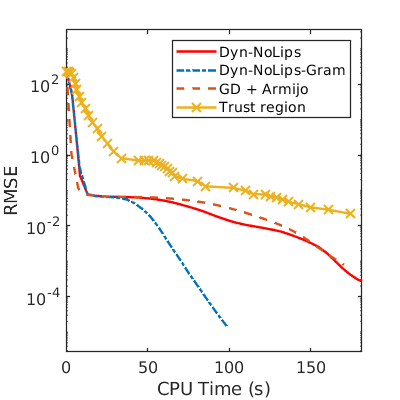}
  }

\caption{Euclidean matrix completion problems on the \texttt{Helix} dataset, with 10\% known distances and two different problem sizes. We present the normalized RMSE over the full distance matrix versus CPU time. The results are averaged over 10 random initializations.}
\label{fig:mc}
\end{figure}

\subsection{Euclidean Distance Matrix Completion}\label{ss:MC}
Euclidean distance matrix completion (EDMC) is the task of recovering the position of $n$ points $x_1^*,\dots,x_n^* \in \reals^r$, given the knowledge of a partial set of pairwise distances $d_{ij} = \|x_i^* - x_j^*\|^2$ for $(i,j) \in \Omega$, where $\Omega \subset [1,n]\times[1,n]$. It is a fundamental problem with applications in sensor network localization and the study of conformation of molecules; see \cite{Fang2012,Qi2013,Dokmanic2015} and references therein.
\newcommand\sumom{\sum_{i,j \in \Omega}}
The Burer-Monteiro nonconvex formulation for solving this problem writes
\begin{equation}\label{eq:edmc}\tag{EDMC}
 \min f(X) := \frac{1}{2}\sum_{(i,j)\in \Omega} \left( \|X_i - X_j\|^2 - d_{ij} \right)^2
\end{equation}
in the variable $X\in \reals^{n \times r}$. It can be rewritten $f(X) = \frac{1}{2} \|\mathcal{P}_\Omega(\kappa(XX^T)-D) \|^2$ where $D$ is the matrix of known distances, $\mathcal{P}_\Omega$ denotes the projection operator such that $\mathcal{P}_\Omega(Y)_{ij} = Y_{ij}$ if $(i,j) \in \Omega$, and $\mathcal{P}_\Omega(Y)_{ij} = 0$ elsewhere, and $\kappa$ is the linear operator defined for $Y \in \reals^{n\times n}$ by
\begin{equation}
    \kappa(Y)_{ij} = Y_{ii} + Y_{jj} - 2 Y_{ij} \,\, \text{for} \,\, 1\leq i,j \leq n
\end{equation}
\paragraph{Applying NoLips with the norm kernel.} Problem \eqref{eq:edmc} falls within our framework with $F(Y) = \frac{1}{2}\|\mathcal{P}_\Omega(\kappa(Y)-D)\|^2$, which can be shown to have a Lipschitz gradient with constant 
\[L_{EDM} :=  9 \max_{i=1\dots n} |\{j| (i,j) \in \Omega\}|.\] 
Therefore, as in the case of SymNMF, the norm kernel $h_N$ can be used with an initial step size $1$ and parameters $\alpha = 6L_{EDM}$ and $\sigma = \frac{1}{3}\|\nabla F(0)\| = 2\|\mathcal{P}_\Omega(D)\|$.

\paragraph{Using the Gram kernel} As the problem is unconstrained, we can also apply minimization using the Gram kernel $h_G$. We use the parameters $\alpha = 2 L_{EDM}$, ${\beta = L_{EDM}}$ and $\sigma = 2 \|\mathcal{P}_\Omega(D)\|$, which ensure that $f$ is $1$-smooth relatively to $h_G$ by Proposition \ref{prop:rel_smoothness_gram}.

{ \paragraph{Computational complexity for NoLips.} As before, the main computational bottleneck for an iteration consists in computing the value and gradient of the objective function. If $p = |\Omega|$ denotes the number of known distances, then the computational complexity is $O(pr)$. If the Gram kernel is used, each iteration requires an additional $O(nr^2 + r^3)$ flops (see Section \ref{ss:gram}), which is negligible compared to the latter in the usual setting where $p >> n$ and $r$ is small.}

\paragraph{Numerical experiments}  We implement the following algorithms: NoLips with a dynamical step size and the norm kernel (\texttt{Dyn-NoLips}), NoLips with a dynamical step size and the Gram kernel (\texttt{Dyn-NoLips-Gram}), gradient descent with Armijo line search (\texttt{GD}), the Riemannian trust region algorithm from \cite{Mishra2011} (\texttt{TR}). We leave out semidefinite relaxations because of their memory requirement which is prohibitive on large data.
As the implementation for \texttt{TR} is provided in Matlab, we run our experiments on Matlab as well, with the same setup as in Section \ref{ss:SymNMF}.

 We try the algorithms on a standard EDMC problem, the 3-dimensional \texttt{Helix} dataset \cite{Mishra2011} which is generated as $X_i = (\cos(3t_i),\sin(3t_i),2t_i)$ where $\{t_i\}_{i=1}^n$ are sampled uniformly in $[0,2\pi]$. We randomly keep only 10 \% on the pairwise distances, and test on two different problem sizes: $n = 2000$ and $n= 5000$. Figure \ref{fig:mc} reports the normalized root mean squared error (RMSE) over \textit{all distances} (known and unknown) averaged on $10$ random initializations. All the algorithms manage to recover the ground truth; the \texttt{Dyn-NoLips-Gram} algorithm shows the best numerical performance, which demonstrates the advantage of using the Gram geometry. 
 
\section{Conclusion}
We proposed a generic approach for solving Burer-Monteiro formulations of low-rank minimization problems using the methodology of Bregman gradient methods and relative smoothness.
We studied two quartic kernels, including a new Gram kernel, and demonstrated their benefits on numerical experiments. In future work, performance could be improved further by studying inertial variants \cite{Mukkamala2019,Hanzely2018}.  New kernels could also be explored beyond the class of quartic functions to tackle other problems with inherent non-Euclidean geometries.

\section*{Code}
The code for reproducing experiments for SymNMF and Euclidean Distance Matrix Completion can be downloaded from the public repository \\{\small \url{https://github.com/RaduAlexandruDragomir/QuarticLowRankOptimization}}

\section*{Acknowledgments}  
{\small
The authors would like to thank the anonymous reviewers for their insightful comments.

Radu-Alexandru Dragomir would like to acknowledge support from an AMX fellowship, the Air Force Office of Scientific Research, Air Force Material Command, USAF, under grant number FA9550-18-1-0226, as well as from S\'ebastien Gadat. 

J\'{e}r\^{o}me Bolte was partially supported by ANR-3IA Artificial and Natural Intelligence Toulouse Institute, and Air Force Office of Scientific Research, and Air Force Material Command, USAF, under grant numbers FA9550-18-1-0226 \& FA9550-19-1-7026.

AA is at CNRS \& d\'epartement d'informatique, \'Ecole normale sup\'erieure, UMR CNRS 8548, 45 rue d'Ulm 75005 Paris, France,  INRIA  and  PSL  Research  University. AA acknowledges support from the French government under management of Agence Nationale de la Recherche as part of the "Investissements d'avenir" program, reference ANR-19-P3IA-0001 (PRAIRIE 3IA Institute), the ML \& Optimisation joint research initiative with the fonds AXA pour la recherche and Kamet Ventures, as well as a Google focused award.}

\appendix
\section*{Appendix}

\section{Solving the Subproblem for Computing the Bregman Iteration Map of the Gram Kernel}
\label{ss:gram_kernel_algo}
While it seems that computing the Bregman iteration map of the Gram kernel involves solving another difficult quartic subproblem, it is actually of small size ($r$ is typically not larger than a few dozens) and can be solved efficiently with the NoLips scheme. 
 
Indeed, the objective function $\phi$ of problem \eqref{eq:gram_subprob_reminder} is 1-smooth relatively to the norm kernel in $\reals^r$ $h_N(x) = \frac{\alpha_u}{4}\|x\|^4 + \frac{\sigma_u}{2}\|x\|^2$ with a choice of parameters $\alpha_u = \alpha + 3 \beta$ and $\sigma_u = \sigma$.

Algorithm \ref{algo:gram_kernel} details the procedure. We initialize $\mu$ with the values for the previous iteration of the outer procedure. This proves to be efficient as the values will not vary much from one iteration to another. For the stopping criterion, we use the scaled gradient norm $\|\nabla \phi(v)\|/\|\eta\|$ and a tolerance value $\epsilon = 10^{-6}$.

The subproblem being very well conditionned, it is minimized easily; in numerical experiments, it usually convergences in no more than $20$ iterations.

\begin{algorithm}[H]
	\begin{algorithmic}
		\STATE {\bfseries Input:} Matrix $X \in \reals^{n \times r}$, gradient of the objective $\nabla f(X)$, step size $\lambda > 0$, parameters $\alpha, \beta,\sigma > 0$, subproblem tolerance $\epsilon$, and (optionally), values $\mu^-$ of $\mu$ computed at the previous iteration.
		\STATE
		\STATE Form $V = \nabla h_G(X) - \lambda \nabla f(X) =  \left(\alpha \|X\|^2 I_r + \beta X^{T} X + \sigma I_r \right) X - \lambda \nabla f(X)$ 
		\STATE Compute $V^T V$
		\STATE Form the eigendecomposition of $V^T V = P^T D P$ where $P \in \mathcal{O}_r$ and $D = \diag(\eta_1^2,\dots,\eta_r^2)$
		\STATE Initialize $\mu$ as $\mu^-$ if provided, and as $(0,\dots,0)$ otherwise.
        \REPEAT
            \STATE Compute $\nabla \phi(\mu)$ where $\nabla \phi(\mu)_i = \alpha \|\mu\|^2 \mu_i + \beta \mu_i^3 + \sigma \mu_i - \eta_i$
            \STATE Compute $\nabla h_N(\mu)$ where $\nabla h_N(\mu)_i = (\alpha + 3 \beta) \|\mu\|^2 \mu_i + \sigma \mu_i$
            \STATE Form $v = \nabla h_N(\mu) - \nabla \phi(u)$
            \STATE Set $\mu \leftarrow \left[\tau_\sigma\left((\alpha + 3 \beta) \|v\|^2\right)\right]^{-1} v$ where $\tau_\sigma$ has been defined in Proposition \ref{prop:mirror_norm}
        \UNTIL{stopping criterion has been satisfied, i.e., $\|\nabla \phi(v)\| / \|\eta\| < \epsilon$}
    \STATE Form $Z = P^T \diag(\mu_1^2,\dots,\mu_r^2) P$
    \STATE Compute $T_\lambda(X) = V \left[\alpha \Tr(Z) I_r + \beta Z + \sigma I_r \right]^{-1} $
    \STATE {\bfseries Output}: Bregman gradient iterate $T_\lambda(X)$
	\end{algorithmic}
	\caption{Computing the Bregman iteration map of the Gram kernel}
	\label{algo:gram_kernel}
\end{algorithm}

\bibliographystyle{unsrt}
\bibliography{library,lib_jerome,datasets}

\begin{thebibliography}{10}

\bibitem{Cand2008}
Emmanuel~J. Cand\`{e}s and Benjamin Recht.
\newblock {Exact Matrix Completion Via Convex Optimization.}
\newblock {\em Found Comput Math}, 9(6), 2009.

\bibitem{Completion2010}
Jian-Feng Cai, Emmanuel~J. Cand{\`{e}}s, and Zuowei Shen.
\newblock {A Singular Value Tresholding Algorithm for Matrix Completion}.
\newblock {\em SIAM Journal on Optimization}, 20(4):1956--1982, 2010.

\bibitem{Jain2012}
Prateek Jain, Praneeth Netrapalli, and Sujay Sanghavi.
\newblock {Low-rank Matrix Completion using Alternating Minimization}.
\newblock In {\em Proceedings of the Forty-fifth Annual ACM Symposium on Theory
  of Computing}, pages 665----674, 2013.

\bibitem{Recht2007}
Benjamin Recht, Maryam Fazel, and Pablo~A. Parrilo.
\newblock {Guaranteed Minimum-Rank Solutions of Linear Matrix Equations via
  Nuclear Norm Minimization}.
\newblock {\em SIAM Review}, 52(3):471--501, 2007.

\bibitem{Mishra2011}
Bandev Mishra, Gilles Meyer, and Rodolphe Sepulchre.
\newblock {Low-rank optimization for distance matrix completion}.
\newblock In {\em Proceedings of the IEEE Conference on Decision and Control},
  pages 4455--4460, 2011.

\bibitem{Fang2012}
Haw~Ren Fang and Dianne~P. O'Leary.
\newblock {Euclidean distance matrix completion problems}.
\newblock {\em Optimization Methods and Software}, 27(4):695--717, 2012.

\bibitem{Candes2015}
Emmanuel~J. Cand{\`{e}}s, Xiaodong Li, and Mahdi Soltanolkotabi.
\newblock {Phase retrieval via wirtinger flow: Theory and algorithms}.
\newblock {\em IEEE Transactions on Information Theory}, 61(4):1985--2007,
  2015.

\bibitem{Chen2015}
Yudong Chen and Martin~J. Wainwright.
\newblock {Fast low-rank estimation by projected gradient descent: General
  statistical and algorithmic guarantees}.
\newblock {\em arXiv preprint arXiv:1509.03025}, 2015.

\bibitem{Burer2005}
Samuel Burer and Renato D~C Monteiro.
\newblock {Local Minima and Convergence in Low-Rank Semidefinite Programming}.
\newblock {\em Mathematical Programming}, 103(3):427--444, 2005.

\bibitem{Tu2015}
Stephen Tu, Ross Boczar, Max Simchowitz, Mahdi Soltanolkotabi, and Benjamin
  Recht.
\newblock {Low-rank Solutions of Linear Matrix Equations via Procrustes Flow}.
\newblock In {\em Proceedings of the 33rd International Conference on
  International Conference on Machine Learning}, pages 964--973, 2016.

\bibitem{Bhojanapalli2015}
Srinadh Bhojanapalli, Anastasios Kyrillidis, and Sujay Sanghavi.
\newblock {Dropping Convexity for Faster Semi-definite Optimization}.
\newblock {\em JMLR: Workshop and Conference Proceedings}, 40:1--53, 2016.

\bibitem{Zhao2015}
Tuo Zhao, Zhaoran Wang, and Han Liu.
\newblock {A Nonconvex Optimization Framework for Low Rank Matrix Estimation}.
\newblock In {\em Advances in Neural Information Processing Systems 28}, pages
  559--567, 2015.

\bibitem{Sun2015}
Ruoyu Sun and Zhi-Quan Luo.
\newblock {Guaranteed Matrix Completion via Nonconvex Factorization}.
\newblock {\em IEEE Transactions on Information Theory}, 62(11):6535--6579,
  2016.

\bibitem{Zheng2016}
Qinqing Zheng and John Lafferty.
\newblock {Convergence Analysis for Rectangular Matrix Completion Using
  Burer-Monteiro Factorization and Gradient Descent}.
\newblock {\em arXiv preprint arXiv:1605.07051}, 2016.

\bibitem{Park2016}
Dohyung Park, Anastasios Kyrillidis, Constantine Caramanis, and Sujay Sanghavi.
\newblock {Finding low-rank solutions to matrix problems , efficiently and
  provably}.
\newblock {\em arXiv preprint arXiv:1606.03168v1}, 2016.

\bibitem{Zheng2015}
Qinqing Zheng and John Lafferty.
\newblock {A Convergent Gradient Descent Algorithm for Rank Minimization and
  Semidefinite Programming from Random Linear Measurements}.
\newblock In {\em Advances in Neural Information Processing Systems 28}, 2015.

\bibitem{Ge2016}
Rong Ge, Jason~D. Lee, and Tengyu Ma.
\newblock {Matrix Completion has No Spurious Local Minimum}.
\newblock {\em Advances in NeuralInformation Processing Systems}, pages
  2973--2981, 2016.

\bibitem{Bauschke2017}
Heinz~H. Bauschke, J{\'{e}}r{\^{o}}me Bolte, and Marc Teboulle.
\newblock {A Descent Lemma Beyond Lipschitz Gradient Continuity: First-Order
  Methods Revisited and Applications}.
\newblock {\em Mathematics of Operations Research}, 42(2):330--348, 2017.

\bibitem{Bolte2018}
J{\'{e}}r{\^{o}}me Bolte, Shoham Sabach, Marc Teboulle, and Yakov Vaisbourd.
\newblock {First order methods beyond convexity and lipschitz gradient
  continuity with applications to quadratic inverse problems}.
\newblock {\em SIAM Journal on Optimization}, 28(3):2131--2151, 2018.

\bibitem{Lin2007}
Chih-Jen Lin.
\newblock {Projected Gradient Methods for Nonnegative Matrix Factorization}.
\newblock {\em Neural Computation}, 2007.

\bibitem{van}
Quang Van~Nguyen.
\newblock Forward-backward splitting with bregman distances.
\newblock {\em Vietnam Journal of Mathematics}, 45(3):519--539, 2017.

\bibitem{Lu2016}
Haihao Lu, Robert~M. Freund, and Yurii Nesterov.
\newblock {Relatively-Smooth Convex Optimization by First-Order Methods, and
  Applications}.
\newblock {\em SIAM Journal on Optimization}, 28(1):333--354, 2018.

\bibitem{Nesterov2004}
Yurii Nesterov.
\newblock {\em {Introductory lectures on convex optimization: A basic course}}.
\newblock Springer US, 2003.

\bibitem{Auslender2006}
Alfred Auslender and Marc Teboulle.
\newblock {Interior gradient and proximal methods for convex and conic
  optimization}.
\newblock {\em SIAM Journal on Optimization}, 16(3):697--725, 2006.

\bibitem{Hanzely2018}
Filip Hanzely, Peter Richt, and Lin Xiao.
\newblock {Accelerated Bregman proximal gradient methods for relatively smooth
  convex optimization}.
\newblock {\em ArXiv preprint arXiv:1808.03045v1}, 2018.

\bibitem{Mukkamala2019}
Mahesh~Chandra Mukkamala, Peter Ochs, Thomas Pock, and Shoham Sabach.
\newblock {Convex-Concave Backtracking for Inertial Bregman Proximal Gradient
  Algorithms in Non-Convex Optimization}.
\newblock {\em arXiv preprint arXiv:1904.03537}, 2019.

\bibitem{Meka2009}
Raghu Meka, Prateek Jain, and Inderjit~S. Dhillon.
\newblock {Guaranteed Rank Minimization via Singular Value Projection}.
\newblock {\em NIPS}, 2010.

\bibitem{Nesterov2007}
Yurii Nesterov.
\newblock {Gradient methods for minimizing composite objective function}.
\newblock {\em CORE Report}, 2007.

\bibitem{Bolte2007}
J{\'{e}}r{\^{o}}me Bolte, Aris Daniilidis, and Adrian Lewis.
\newblock {The {\L}ojasiewicz Inequality for Nonsmooth Subanalytic Functions
  with Applications to Subgradient Dynamical Systems}.
\newblock {\em SIAM Journal on Optimization}, 17(4):1205--1223, 2007.

\bibitem{Ding2005}
Chris Ding, Xiaofeng He, and Horst Simon.
\newblock {On the Equivalence of Nonnegative Matrix Factorization and Spectral
  Clustering}.
\newblock In {\em Proceedings of the 2005 SIAM ICDM}, number~4, pages 126--135,
  2005.

\bibitem{He2011}
Zhaoshui He, Shengli Xie, Rafal Zdunek, Guoxu Zhou, and Andrzej Cichocki.
\newblock {Symmetric nonnegative matrix factorization: Algorithms and
  applications to probabilistic clustering}.
\newblock {\em IEEE Transactions on Neural Networks}, 22(12):2117--2131, 2011.

\bibitem{Kuang2015}
Da~Kuang, Sangwoon Yun, and Haesun Park.
\newblock {SymNMF: nonnegative low-rank approximation of a similarity matrix
  for graph clustering}.
\newblock {\em Journal of Global Optimization}, 62(3):545--574, 2015.

\bibitem{Controller2013}
Jingu Kim and Haesun Park.
\newblock {Fast Nonnegative Matrix Factorization: An Active-set-like Method and
  Comparisons}.
\newblock {\em SIAM Journal on Scientific Computing}, 33(6):3261--3281, 2013.

\bibitem{Cichocki2009}
Andrzej Cichocki and Anh~Huy Phan.
\newblock {Fast local algorithms for large scale nonnegative matrix and tensor
  factorizations}.
\newblock {\em IEICE Transactions on Fundamentals of Electronics,
  Communications and Computer Sciences}, 2009.

\bibitem{Vandaele2016}
Arnaud Vandaele, Nicolas Gillis, Qi~Lei, Kai Zhong, and Inderjit Dhillon.
\newblock {Efficient and non-convex coordinate descent for symmetric
  nonnegative matrix factorization}.
\newblock {\em IEEE Transactions on Signal Processing}, 64(21):5571--5584,
  2016.

\bibitem{Lu2017}
Songtao Lu, Mingyi Hong, and Zhengdao Wang.
\newblock {A Nonconvex Splitting Method for Symmetric Nonnegative Matrix
  Factorization : Convergence Analysis and Optimality}.
\newblock {\em IEEE Transactions on Signal Processing}, 65(12):2572--2576,
  2017.

\bibitem{zhu2018a}
Zhihui Zhu, Xiao Li, Kai Liu, and Qiuwei Li.
\newblock {Dropping Symmetry for Fast Symmetric Nonnegative Matrix
  Factorization}.
\newblock In {\em Advances in Neural Information Processing Systems 31}, 2018.

\bibitem{Bezanson2017}
Stefan~Karpinski {Jeff Bezanson, Alan Edelman} and Viral~B. Shah.
\newblock {Julia : A Fresh Approach to Numerical Computing}.
\newblock {\em SIAM Review}, 59(1):65--98, 2017.

\bibitem{CWH09}
Deng Cai, Xuanhui Wang, and Xiaofei He.
\newblock Probabilistic dyadic data analysis with local and global consistency.
\newblock In {\em Proceedings of the 26th Annual International Conference on
  Machine Learning (ICML'09)}, pages 105--112, 2009.

\bibitem{CMHZ08}
Deng Cai, Qiaozhu Mei, Jiawei Han, and Chengxiang Zhai.
\newblock Modeling hidden topics on document manifold.
\newblock In {\em Proceeding of the 17th ACM conference on Information and
  knowledge management (CIKM'08)}, pages 911--920, 2008.

\bibitem{CHZH07}
Deng Cai, Xiaofei He, Wei~Vivian Zhang, and Jiawei Han.
\newblock Regularized locality preserving indexing via spectral regression.
\newblock In {\em Proceedings of the 16th ACM conference on Conference on
  information and knowledge management (CIKM'07)}, pages 741--750, 2007.

\bibitem{CHH05}
Deng Cai, Xiaofei He, and Jiawei Han.
\newblock Document clustering using locality preserving indexing.
\newblock {\em IEEE Transactions on Knowledge and Data Engineering},
  17(12):1624--1637, December 2005.

\bibitem{Zhu2018b}
Zhihui Zhu, Xiao Li, Kai Liu, and Qiuwei Li.
\newblock {Dropping Symmetry for Fast Symmetric Nonnegative Matrix
  Factorization}.
\newblock {\em NIPS}, 2018.

\bibitem{Qi2013}
Hou~Duo Qi and Xiaoming Yuan.
\newblock {Computing the nearest Euclidean distance matrix with low embedding
  dimensions}.
\newblock {\em Mathematical Programming}, 147(1-2):351--389, 2013.

\bibitem{Dokmanic2015}
Ivan Dokmanic, Reza Parhizkar, Juri Ranieri, and Martin Vetterli.
\newblock {Euclidean Distance Matrices: Essential theory, algorithms, and
  applications}.
\newblock {\em IEEE Signal Processing Magazine}, 32(6):12--30, 2015.

\end{thebibliography}

\end{document}